\documentclass{amsart}

 \usepackage{amssymb,latexsym,amsmath,amsthm}

\usepackage[numbers,sort&compress]{natbib}

 \renewcommand{\div}{\mathop{\mathrm{div}}\nolimits}

\newtheorem*{thm*}{Theorem A}

\newtheorem{thm}{Theorem}[section]

\newtheorem{dfn}{Definition}[section]

\newtheorem{lemma}{Lemma}[section]

\newtheorem{remark}{Remark}

\begin{document}

\def\IR{{\mathbb{R}}}

\title{Rigidity results for stable solutions of symmetric systems}
\maketitle

\begin{center}
\author{Mostafa Fazly\footnote{The author is pleased to acknowledge the support of University of Alberta Start-up Grant RES0019810 and National Sciences and Engineering Research Council of Canada (NSERC) Discovery Grant RES0020463.}
}
\\
{\it\small Department of Mathematical and Statistical Sciences, 632 CAB, University of Alberta}\\
{\it\small Edmonton, Alberta, Canada T6G 2G1}\\
{\it\small e-mail: fazly@ualberta.ca}\vspace{1mm}
\end{center}

\begin{abstract}  We study stable solutions of the following nonlinear system  $$  -\Delta u    =   H(u)  \quad  \text{in} \ \  \Omega$$ where $u:\IR^n\to \IR^m$,  $H:\IR^m\to \IR^m$ and $\Omega$ is a domain in $\mathbb R^n$.  We introduce  the novel notion of symmetric systems. The above system is said to be symmetric if the matrix of gradient of all components of $H$ is symmetric. It seems that this concept is crucial to prove Liouville theorems, when $\Omega=\mathbb R^n$, and regularity results, when $\Omega=B_1$, for stable solutions of the above system for a general  nonlinearity $H \in C^1(\mathbb R ^m)$.    Moreover, we provide an improvement for a linear Liouville theorem given in \cite{fg} that is a key tool to establish De Giorgi type results in lower dimensions for elliptic equations and systems. 
\end{abstract}


\section{Introduction} 

We examine the following semilinear elliptic system of equations 
 \begin{equation} \label{main}
   -\Delta u    =   H(u)  \quad  \text{in} \ \  \mathbb{R}^n 
  \end{equation}   
  where $u:\IR^n\to \IR^m$ and $H:\IR^m\to \IR^m$.     We use the notation $u=(u_1,\cdots,u_m)$,  $H(u)=(H_1(u),\cdots,H_m(u))$ and  $\partial_j H_i(u)=\frac{\partial H_i(u)}{\partial {u_j}} $ where $\partial_i H_j (u)\partial_j H_i (u)\ge 0$ for $1\le i<j\le m$.         We are interested in the qualitative properties of radial stable solutions of (\ref{main}) when $H\in C^1(\mathbb R^m)$ is a general function.    Here is the notion of stability.      
\begin{dfn} \label{dfnstable}
A solution $u=(u_i)_{i}$ of (\ref{main}) is said to be stable when there is positive solution $\zeta=(\zeta_i)_i $ for the following linearized system  
\begin{equation} \label{sta}
 -\Delta \zeta_i =   \sum_{j=1}^{n} \partial_j H_i(u) \zeta_j  \quad  \text{in}  \ \  \mathbb{R}^n,
 \end{equation} 
 for all $ i=1,\cdots,m$. 
\end{dfn} 

For the case of $m=1$, equation (\ref{main}) turns into the scalar equation that is studied extensively in the literature.    As it is shown by Dupaigne and Farina in \cite{df},  any classical bounded stable solution of (\ref{main}) is constant provided $1 \le n \le 4$ and $0\le H \in C^1(\mathbb R)$ is a general nonlinearity. For particular nonlinearities $H(u)=e^u$ and $H(u)=u^p$ where $p>1$ the differential equation (\ref{main}) is called Gelfand and Lane-Emden equations, respectively,  and optimal Liouville theorems are provided by Farina in  \cite{f1,f2}.     In addition, when $H(u)=-u^{-p}$ for $p>0$ the equation (\ref{main}) is known as the Lane-Emden equation with negative exponent nonlinearity and optimal Liouville theorems are given by Esposito-Ghoussoub-Guo in \cite{egg1,egg2}.  Critical dimensions for Liouville theorems are 
\begin{itemize}
\item  $1\le n <10$, 
\item  $1\le n < 2+ \frac{4}{p-1}(p+\sqrt{p(p-1)})$ 
\item  $1\le n < 2+ \frac{4}{p+1}(p+\sqrt{p(p+1)})$ 
\end{itemize}
for the equations of Gelfand, Lane-Emden and Lane-Emden with negative exponent nonlinearity, respectively. Note that these dimensions are much higher than the fourth dimensions that is known for the case of general nonlinearity.    Let us mention that various equations with a singular nonlinear term in the case of singular equations have been studied in the book of Ghergu and Radulescu \cite{gr}. 

For radial solutions,  it is proved by Cabr\'{e}-Capella \cite{cc,CC} and Villegas \cite{sv} that any bounded radial stable solution of (\ref{main})  has to be constant provided $1\le n < 10$ when $H\in C^1(\mathbb R)$ is a general nonlinearity.  This is an optimal Liouville theorem. 

Note that for the case of systems, that is when $m \ge 1$, a counterpart of the Dupaigne-Farina's Liouville theorem is given by Ghoussoub and the author  in \cite{fg} for gradient systems that is where  there exists a $\mathcal H: \mathbb R^m\to \mathbb R$ such that $H=\nabla \mathcal H$.  It is in fact shown that if all components of $H$ are nonnegative, then any bounded stable solution of the system (\ref{main}) is necessarily constant provided $1\le n\le 4$.        

 In this paper, regarding Liouville theorems, we first provide an extension for a linear Liouville theorem that is proved by Ghoussoub and the author in \cite{fg}.  This Liouville theorem is a key tool in proving De Giorgi type results in lower dimensions for both systems and scalar equations.    Then we apply this Liouville theorem to establish a Liouville theorem  in lower dimensions $1\le n\le 4$  for elliptic symmetric systems (\ref{main})  that is when  the matrix of gradient of all components of $H$ is symmetric.   In addition,     for radial stable solutions,     we prove Liouville theorems and pointwise estimates for  elliptic system (\ref{main}) where $H=(H_i)_i$ for $H_i \in C^1(\mathbb R^m)$, $1\le i \le m$,  is a general nonlinearity. As in the scalar case, the critical dimension is $n=10$.

Roughly speaking, there is a correspondence between the regularity of stable solutions on bounded domains and the Liouville theorems for stable solutions on $\mathbb R^n$, via rescaling and a blow up procedure.     Consider a counterpart of system (\ref{main}) with the Dirichlet boundary conditions 
\begin{eqnarray} \label{mainom} 
  \left\{ \begin{array}{lcl}
\hfill   -\Delta u    &=& \Lambda  H(u)  \quad \text{in} \ \  \Omega \\
\hfill u &=&0  \qquad \text{on} \ \ \partial \Omega
\end{array}\right.
  \end{eqnarray}   
  where  $\Lambda=(\lambda_i)_i$ is a positive sequence of parameters and  $\Omega$ is a bounded domain in $\IR^n$.   Similarly, a solution $u$ of (\ref{mainom}) is said to be a stable solution if the linearized operator has a positive first eigenvalue.  The regularity of stable solutions depends on the dimension $n$,  domain $\Omega$ and also the nonlinearity $H$.   We refer the interested readers to the work of Montenegro \cite{Mont} for the notion of stability for the case of $m=2$. 
  
  For the case of $m=1$ and for explicit nonlinearities $H(u)=e^u$, $H(u)=(1+u)^p$ where $p>1$ and $H(u)=(1-u)^{-p}$ where $0<u<1$ and $p>0$, regularity of stable solutions and extremal solutions are now quite well understood, see for instance \cite{bcmr,bv,Cabre,cc,CC,cr,CR,egg1,egg2,Nedev,v,sv,SV12} and references therein.    It is well known that there exists a critical parameter  $ \Lambda^* \in (0,\infty)$, called the extremal parameter,  such that for all $ 0<\Lambda < \Lambda^*$ there exists a smooth, minimal solution $u_\Lambda$ of (\ref{mainom}).  Here the minimal solution means in the pointwise sense.  In addition for each $ x \in \Omega$ the map $ \Lambda \mapsto u_\Lambda(x)$ is increasing in $ (0,\Lambda^*)$.   This allows one to define the pointwise limit $ u^*(x):= \lim_{\Lambda \nearrow \Lambda^*} u_\Lambda(x)$  which can be shown to be a weak solution, in a suitably defined sense, of (\ref{mainom}).  For this reason $ u^*$ is called the extremal solution.      It is also known that for $ \Lambda >\Lambda^*$ there are no weak solutions of (\ref{mainom}).  Also one can show the minimal solution $ u_\Lambda$  is a stable solution of (\ref{mainom}).     Consider a general nonlinearity $H \in C^1(\mathbb R)$ that satisfies 
 \[ (R) \;\;  \mbox{$H$ is smooth, increasing and convex with $ H(0)=1 $ and $ H$ superlinear at $ \infty$.}\]    Brezis and V\'{a}zquez \cite{bv} raised the question of determining the boundedness of $u^*$, for general nonlinearities $H$ that satisfies (R).  Nedev in \cite{Nedev} showed that the extremal solution  of (\ref{main}) is bounded provided $1\le n\le 3$ when general nonlinearity $H$ satisfies (R).  The best known result on the regularity of extremal solutions for a general nonlinearity $H$ that satisfies (R) (no convexity on $H$ is imposed) was established by Cabr\'{e} in \cite{Cabre}  via geometric-type Sobolev inequalities provided $1 \le n \le 4$ and $\Omega$ is a convex domain. In this dimension, the convexity of the domain $\Omega$ was relaxed by Villegas in \cite{SV12}, however convexity of $H$ was assumed.  We also refer interested readers to \cite{cr} where regularity of stable solutions are proved up to seven dimensions  in domains of double revolution.   
      
 For the case of systems, i.e. $m\ge1$, Cowan and the author in \cite{cf} proved that the extremal solution of (\ref{main}) when $\Omega$ is a convex domain is regular provided $1\le n \le 3$, $m=2$ and $H_1(u_1,u_2)=f'(u_1)g(u_2)$ and  $H_2(u_1,u_2)=f(u_1)g'(u_2)$  for general nonlinearities $f,g\in C^1(\mathbb R)$ that satisfy (R).   This can be seen as a counterpart of the Nedev's result for elliptic gradient systems.  For explicit nonlinearities $ f(u_1)=(u_1+1)^p,  g(u_2)=(u_2+1)^q$ where $ p,q>2$, regularity of extremal solution is provided in dimensions $$ 1 \le n < 2 + \frac{4}{p+q-2} \max\{ t_+(p-1),t_+(q-1)\} \ \text{where} \  t_+(\alpha)= \alpha+ \sqrt{\alpha(\alpha-1)}.$$ For the Gelfand system,  regularity of the extremal solutions is given by Cowan in \cite{cowan} and by Dupaigne-Farina-Sirakov in \cite{dfs}.   For radial solutions,  it is also shown in \cite{cf} that stable solutions are regular in dimensions $1\le n <10$ when  $m=2$ and $H_1(u_1,u_2)=f'(u_1)g(u_2)$ and  $H_2(u_1,u_2)=f(u_1)g'(u_2)$ for general nonlinearities $f,g\in C^1(\mathbb R)$ that satisfy (R). This is a counterpart of the regularity result of Cabr\'{e}-Cappella \cite{cc} and Villegas \cite{sv} for elliptic gradient systems.

 Regarding regularity results, we provide an extension of the regularity results given by Cowan and the author in \cite{cf} to symmetric systems of the form (\ref{main}) where  $H=(H_i)_i$ for each $H_i \in C^1(\mathbb R^m)$ is a general nonlinearity.       In the next section, we state the      notion of symmetric systems. Then in Section \ref{sLiou},  we provide Liouville theorems for system (\ref{main}) and also the linearized system. Finally, in Section \ref{sReg}, we shall prove regularity results for elliptic system (\ref{main}).

%%%%%%%%%%%%%%%%%%%%%%%%%%%%%%%%
 
\section{The notion of symmetric systems}\label{secstable}

Here is the notion of the symmetric systems. 
\begin{dfn}\label{symmetric} We call system (\ref{main}) symmetric if the matrix of gradient of all components of $H$ that is $$\mathbb{H}:=(\partial_i H_j(u))_{i,j=1}^{m}$$  is symmetric. 
   \end{dfn}
Note that when $m=1$ then system (\ref{main}) is clearly symmetric.    Let us start with the following  stability inequality that plays an important role in this paper.  For the case of systems, see \cite{fg, dfs, cf} for similar stability inequalities on $\mathbb R^n$ and on a bounded domain $\Omega$.  
  \begin{lemma}\label{stabineq}  
  Let $u$ denote a stable solution of (\ref{main}).  Then 
\begin{equation} \label{stability}
\sum_{i,j=1}^{m} \int  \sqrt{\partial_j H_i(u) \partial_i H_j(u)} \phi_i \phi_j \le \sum_{i=1}^{m} \int  |\nabla \phi_i|^2, 
\end{equation} 
for any $\phi=(\phi_i)_i^m$ where $ \phi_i \in L^{\infty} (\mathbb R^n) \cap H^1(\mathbb R^n)$ with compact support and $1\le i\le m$. 
\end{lemma}  

\begin{proof}   From Definition \ref{dfnstable},  there is a sequence $\zeta=(\zeta_i)_i^m$ such that $ 0 < \zeta_i$ and 
\[  -\Delta \zeta_i =  \sum_{j=1}^{n} \partial_j H_i(u) \zeta_j \ \ \text{for all} \ \  i=1,\cdots,m.  \] 
Consider test function  $\phi=(\phi_i)_i^m$ where $ \phi_i \in L^{\infty} (\mathbb R^n) \cap H^1(\mathbb R^n)$ with compact support  and multiply both sides of the above inequalities with $\frac{\phi_i^2}{\zeta_i}$ to obtain
\begin{eqnarray*}
  \sum_{j=1}^{n}  \int \partial_j H_i(u) \zeta_j \frac{\phi_i^2}{\zeta_i} \le \int   - \frac{\Delta \zeta_i}{\zeta_i} {\phi_i^2}.
  \end{eqnarray*}
Note that from the  Young's inequality it is straightforward to see $$\int   - \frac{\Delta \zeta_i}{\zeta_i} {\phi_i^2} \le \int |\nabla \phi_i|^2 \ \ \text{for all} \ \  i=1,\cdots,m  .  $$
On the other hand, we have 
\begin{eqnarray*}
  \sum_{i,j=1}^{m}  \int \partial_j H_i(u) \zeta_j \frac{\phi_i^2}{\zeta_i}& =&   \sum_{i<j}^{m}  \int \partial_j H_i(u) \zeta_j \frac{\phi_i^2}{\zeta_i} + \sum_{i>j}^{n}  \int \partial_j H_i(u) \zeta_j \frac{\phi_i^2}{\zeta_i} + \int \partial_i H_i(u) {\phi_i^2}\\&=&
  \sum_{i<j}^{m}  \int \partial_j H_i(u) \zeta_j \frac{\phi_i^2}{\zeta_i} + \sum_{i<j}^{m}  \int \partial_i H_j(u) \zeta_i \frac{\phi_j^2}{\zeta_j} +\int \partial_i H_i(u) {\phi_i^2}
  \\&=& \sum_{i<j}^{m}  \int  \left( \partial_j H_i(u) \zeta_j \frac{\phi_i^2}{\zeta_i} + \partial_i H_j(u) \zeta_i \frac{\phi_j^2}{\zeta_j}  \right)+ \int \partial_i H_i(u) {\phi_i^2}
  \\&\ge & 2 \sum_{i<j}^{m}  \int   \sqrt{\partial_j H_i(u) \partial_i H_j(u)} \phi_i \phi_j+ \int \partial_i H_i(u) {\phi_i^2}
  \\&=&\sum_{i,j=1}^{m}  \int \sqrt{\partial_j H_i(u) \partial_i H_j(u)} \phi_i\phi_j.
  \end{eqnarray*}
This finishes the proof. 
\end{proof}  
For radial solutions of elliptic systems (\ref{main}), the following stability inequality holds. 
  
\begin{lemma} \label{stabineqgen}  
  Suppose that $u$ is a radial stable solution of  (\ref{main}).  Then 
\begin{eqnarray} \label{stabindep}
 &&(n-1) \sum_{i=1}^{m} \int_{\mathbb R^n}  \frac{u'^2_i(|x|)}{|x|^2} \phi^2(x) dx \le \sum_{i=1}^{m} \int_{\mathbb R^n}    u'^2_i(|x|)   |\nabla \phi(x)|^2 dx  \\ &
& \label{tail}+    \sum_{i,j=1}^{m} \int_{\mathbb R^n}  \left(  \partial_j H_i(u) -\sqrt{\partial_j H_i(u) \partial_i H_j(u) } \right)  u'_i (|x|) u'_j (|x|) \phi^2(x) dx
\end{eqnarray} 
   for all $ \phi \in L^{\infty} (\mathbb R^n) \cap H^1(\mathbb R^n)$ with compact support. 

\end{lemma}

\begin{proof}  Taking derivative of (\ref{main}) with respect to $r$ gives 
 \begin{equation}\label{pderadial}
   -\Delta u'_i +\frac{n-1}{r^2} u'_i   =  \sum_{j=1}^m  \partial_j H_i(u) u'_j \quad \text{for} \quad 0<r<1 \quad\text{and} \quad i=1,\cdots,m. 
  \end{equation} 
Multiply the $i^{th}$ equation of (\ref{pderadial}) with $u'_i \phi^2$   for all $ \phi \in L^{\infty} (\mathbb R^n) \cap H^1(\mathbb R^n)$ with compact support gives 
\begin{equation}\label{integralradial}
\int  |\nabla u'_i|^2\phi^2 +\frac{1}{2} \nabla {u'_i}^2\cdot\nabla\phi^2 +\frac{n-1}{r^2} {u'_i}^2\phi^2   = \int  \sum_{j=1}^m  \partial_j H_i(u) {u'}_ju'_i \phi^2
  \end{equation}
  for all $0<r<1$ and  $ i=1,\cdots,m$.  On the other hand, testing (\ref{stability}) on $\phi_i = u'_i\phi$  where  $\phi$ is the same test function as above then   we get 
  \begin{equation} \label{stabsub}
  \sum_{i,j=1}^{m} \int  \sqrt{\partial_j H_i(u) \partial_i H_j(u)} u'_i u'_j\phi^2   \le \sum_{i=1}^{m}  \int  |\nabla ( u'_i\phi)|^2 . 
\end{equation} 
  Expanding the right-hand side we get 
  \begin{equation}\label{stabuphi}
  \sum_{i=1}^{m}  \int  |\nabla ( u'_i\phi)|^2 =  \sum_{i=1}^{m}   \int  |\nabla u'_i|^2\phi^2+ {u'_i}^2|\nabla \phi|^2 +\frac{1}{2} \nabla \phi^2\cdot\nabla {u'_i}^2. 
  \end{equation}
  From (\ref{integralradial}) we get the following equality for part of the right-hand side of  (\ref{stabuphi})
 \begin{equation} \label{rhs}
  \sum_{i=1}^{m}  \int  |\nabla u'_i|^2\phi^2 +\frac{1}{2} \nabla \phi^2\cdot\nabla {u'_i}^2 = \sum_{i,j=1}^m \int \partial_j H_i(u) {u'}_ju'_i \phi^2 - \sum_{i=1}^{m} \int \frac{n-1}{r^2} {u'_i}^2\phi^2  . 
    \end{equation}
Now from (\ref{stabsub}),  (\ref{stabuphi}) and (\ref{rhs})  we get 
  \begin{eqnarray*}
  \sum_{i,j=1}^{m} \int  \sqrt{\partial_j H_i(u) \partial_i H_j(u)} u'_i u'_j\phi^2   &\le&  \sum_{i=1}^{m}   \int   {u'_i}^2|\nabla \phi|^2 +\sum_{i,j=1}^m \int  \partial_j H_i(u) {u'}_ju'_i \phi^2 \\&& - \sum_{i=1}^{m} \int \frac{n-1}{r^2} {u'_i}^2\phi^2  .
  \end{eqnarray*}
   
\end{proof}  
\begin{remark} Note that for symmetric systems the tail of the inequality (\ref{stabindep}) that is (\ref{tail}) vanishes.  Therefore, the nonlinearity $H$ does not appear in the stability inequality for symmetric systems.  Applying this inequality to a radial test function $\phi(|x|)$ one can see that 
  \begin{equation} \label{Rstablesym}
(n-1) \sum_{i=1}^{m} \int_{0}^{\infty}  {u'_i}^2(t) \phi^2(t) t^{n-3} dt \le \sum_{i=1}^{m} \int_{0}^\infty    {u'_i}^2(t)   \phi'(t)^2 t^{n-1} dt  
\end{equation} 
where $\phi \in L^{\infty} (\mathbb R^+) \cap H^1(\mathbb R^+)$ is a compactly supported test function.  
 \end{remark}
 
 %%%%%%%%%%%%%%%%%%%%%%%
 
\section{Liouville theorems for symmetric systems}\label{sLiou}
In this section, we provide Liouville theorems for a linearized elliptic system associated to (\ref{main}), then we establish an optimal Liouville theorem for radial stable solutions of (\ref{main}) with a general nonlinearity $H$.  

\subsection{A Liouville theorem for the linearized system} Suppose that $u$ is a $H$-monotone solution of symmetric system (\ref{main}). A solution $u=(u_k)_{k=1}^m$ of (\ref{main}) is said to be {\it $H$-monotone} if the following hold:
\begin{enumerate}
 \item[(i)] For every $1\le i \le m$, each $u_i$ is strictly monotone in the $x_n$-variable (i.e., $\partial_n u_i\neq 0$).

\item[(ii)]  For $i\le j$,  we have 
  \begin{equation}
\hbox{$\partial_j H_i(u) \partial_n u_i(x) \partial_n u_j (x)\ge 0$  for all $x\in\mathbb {R}^n$.}
\end{equation}
\end{enumerate}
See \cite{fg} for more details.   Let $\phi_i := \partial _n u_i$ and $\psi_i:=\nabla u_i\cdot\eta$ for any fixed $\eta=(\eta',0)\in \mathbb{R}^{n-1}\times\{0\}$ in such a way 
that $\sigma_i:=\frac{\psi_i}{\phi_i}$.  Then $(\phi_i)_i$ and $(\psi_i)_i$ satisfy (\ref{sta}).  Straightforward calculations show that for $H$-monotone solutions we have
\begin{eqnarray*}
\sum_{i=1}^m \sigma_i\div(\phi_i^2 \nabla \sigma_i) &=& \sum_{i,j} \phi_i \phi_j \partial_j H_i(u) \sigma_i (\sigma_i-\sigma_j)\\&=& \sum_{i< j}   \phi_i \phi_j \partial_j H_i(u)  \sigma_i (\sigma_i-\sigma_j) + \sum_{i> j}  \phi_i \phi_j \partial_j H_i(u)  \sigma_i  (\sigma_i-\sigma_j)  \\&=&\sum_{i< j}  \phi_i \phi_j \partial_j H_i(u)   \sigma_i (\sigma_i-\sigma_j)  + \sum_{i< j}  \phi_i \phi_j \partial_j H_i(u)  \sigma_j (\sigma_j-\sigma_i)  \\&=&\sum_{i< j}  \phi_i \phi_j \partial_j H_i(u)  (\sigma_i-\sigma_j)^2 \ge 0. 
  \end{eqnarray*}
In what follows we prove a Liouville theorem for this differential inequality. Note that for $m=1$ this type of Liouville theorem was noted by Berestycki, Caffarelli and Nirenberg in \cite{bcn} and used by Ghoussoub-Gui \cite{gg1} and later by Ambrosio and Cabr\'{e} \cite{ac}  to prove the De Giorgi's conjecture \cite{DeGiorgi} in dimensions two and three. See also \cite{mos}.  For the case of $m\ge 1$ this improves a linear Liouville theorem that is proved by Ghoussoub and the author in \cite{fg} and applied to establish De Giorgi type results for elliptic systems. Note that the proof of the De Giorgi's conjecture for a general nonlinearity provided by Ghoussoub-Gui \cite{gg1} and Alberti, Ambrosio and Cabr\'{e}  \cite{aac} for dimensions two and three, respectively. 

 Consider the set of functions with a limited growth at infinity as $$\mathcal F=\left\{F:\mathbb R^+\to\mathbb R^+, F \  \text{is nondecreasing and} \ \int_{2}^{\infty} \frac{dr}{rF(r)}=\infty\right\}.$$
In particular, $F(r)=\log r$ belongs to this class and $F(r)=r$ does not belong to $\mathcal F$. As far as we know,  this class of functions was considered by Karp \cite{k1,k2} for the first time. Here is the Liouville theorem. 

\begin{thm}\label{liouville} Assume that each $(\phi_i)_{i=1}^{m} \in L^{\infty}_{loc}(\mathbb {R}^n) $ does not change sign in $\mathbb{R}^n$ and $(\sigma_i)_{i=1}^{m} \in H^1_{loc}(\mathbb{R}^n)$ is such that
\begin{equation}\label{liouassum}
\limsup_{R\to \infty} \frac{1}{R^2F(R)}\sum_{i=1}^{m}\int_{B_{2R}\setminus B_R}  \phi_i^2\sigma_i^2< \infty,
 \end{equation}
for some $F\in \mathcal{F}$. Suppose also that $(\sigma_i)_i$ is a solution of 
\begin{eqnarray}
\label{div}
\sum_{i=1}^m \sigma_i\div(\phi_i^2 \nabla \sigma_i) \ge 0   \ \ \text{in}\ \ \mathbb{R}^n. 
  \end{eqnarray}
 Then, for all  $i=1,...,m$, the functions $\sigma_i$ are constant.
\end{thm}
\begin{proof}  Since $(\sigma_i)_i$ satisfies (\ref{div}), straightforward calculations show that 
\begin{eqnarray} \label{}
\div(\phi_i^2 \sigma_i \nabla \sigma_i) = |\nabla \sigma_i|^2\phi_i^2+ \sigma_i\div(\phi_i^2 \nabla \sigma_i).  
 \end{eqnarray}
Therefore, 
\begin{eqnarray}
 \sum_{i=1}^m |\nabla \sigma_i|^2\phi_i^2  \le \sum_{i=1}^m \div(\phi_i^2 \sigma_i \nabla \sigma_i)   . \end{eqnarray}
Integrating both sides we get 
\begin{eqnarray*}
\sum_{i} \int_{B_R} |\nabla \sigma_i|^2\phi_i^2  &\le&  \sum_{i} \int_{B_R} \div(\phi_i^2 \sigma_i \nabla \sigma_i) = \sum_{i}  \int_{\partial B_R} \phi_i^2 \sigma_i \nabla \sigma_i\cdot\eta
\\&\le &  \sum_{i} \int_{\partial B_R} \phi_i^2 |\sigma_i| |\nabla \sigma_i|\\
&\le&  \left(\sum_{i} \int_{\partial B_R} (\phi_i\sigma_i)^2\right)^\frac{1}{2} \left(\sum_{i} \int_{\partial B_R}    |\nabla \sigma_i|^2 \phi_i^2 \right)^\frac{1}{2}.
   \end{eqnarray*}
If all $\sigma_i$ for $i=1,\cdots,m$ are not constant, then there exists $R_0 > 0$ such that $D(R) > 0$ for every $R > R_0$ and 
\begin{eqnarray}\label{D}
D(R) \le D'(R)^\frac{1}{2} \left(\int_{\partial B_R} \sum_{i}   (\phi_i\sigma_i)^2  \right)^\frac{1}{2},
   \end{eqnarray}
where $D(R):=\sum_{i} \int_{B_R} |\nabla \sigma_i|^2\phi_i^2 $. Integrating (\ref{D}) and using the Schwarz inequality we get that for $r_2>r_1>R_0$, 
\begin{eqnarray*}
(r_2-r_1)^2  \left(  \int_{B_{r_2}\setminus B_{r_1}} \sum_{i}   (\phi_i\sigma_i)^2    \right)^{-1}  &\le&
\int_{r_1}^{r_2} \left(\int_{\partial {B_R}} \sum_{i}   (\phi_i\sigma_i)^2    \right)^{-1} dR\\&\le&  \int_{r_1}^{r_2} \frac{D'(R)}{D^2(R)}dR \\&=&\frac{1}{D(r_1)}-\frac{1}{D(r_2)}.
  \end{eqnarray*}

Now, take $r_2=2^{k+1}r_0$ and $r_1=2^{k}r_0$ for fixed $r_0>R_0$ and $k\ge 0$.  From (\ref{liouassum}) we get that $D(r_0)=0$ for $r_0>R_0$ which is a contradiction.

\end{proof}

Here is a Liouville theorem that is an application of  Theorem \ref{liouville}.     

\begin{thm} Suppose that $u=(u_i)_i$ is bounded stable solution of symmetric system (\ref{main}) where $H=(H_i)_i$ for $0\le H_i\in C^1(\mathbb R^m)$ and $m\ge 1$.  Then each $u_i$ is constant provided $1\le n\le 4$. 
\end{thm}
\begin{proof}
Multiply both sides of system (\ref{main}) with $(u_i-||u_i||_{\infty})\phi^2$ where $\phi$ is a test function. Since $ H_{i}(u) (u_i-||u_i||_{\infty})\le 0$ we have 
\begin{eqnarray}
\label{supersol}
\hfill -\Delta u_i (u_i-||u_i ||_\infty)\phi^2 \le 0  \ \ \text{in}\ \ \mathbb {R}^n.
  \end{eqnarray}
After an integration by parts, we end up with 
\begin{eqnarray}
\label{}
\int_{B_R} |\nabla  u_i|^2\phi^2\le 2 \int_{B_R} |\nabla u_i||\nabla \phi |(||u_i ||_\infty-u_i)\phi \ \ \text{for all $1\le i\le m$.}
  \end{eqnarray}
Using Young's inequality and adding we get 
\begin{equation}\label{decayn-2}
\sum_{i=1}^m \int_{B_R} |\nabla  u_i|^2 \le R^{n-2}.
  \end{equation}
Now one can apply Theorem \ref{liouville} to quotients of partial derivatives to obtain that each $u_i$ is one dimensional solutions as long as $n\le 4$. Note that $u_i$ is a bounded solution for (\ref{supersol}) in dimension one, and the corresponding decay estimate (\ref{decayn-2}) now implies that  $u_i$ must be constant for all $1\le i\le m$.  

\end{proof}
Inspired by De Giorgi type results given in \cite{fg}, we have the following immediate consequence of Theorem \ref{liouville}.  

\begin{thm}
 Suppose that $u=(u_i)_i$ is a $H$-monotone solution of symmetric orientable system (\ref{main}) where $H=(H_i)_i$ for $H_i\in C^1(\mathbb R^m)$ and $m\ge 1$.  Then each $u_i$ is a one dimensional function provided $1\le n\le 3$. 
\end{thm}
\begin{proof} We omit the proof, since it is closely related to the proofs provided in \cite{fg}. 
\end{proof}

\subsection{Nonlinear Liouville theorems for radial solutions}

For radial stable solutions of symmetric systems (\ref{main}) the following Liouville theorem and pointwise estimates hold. Note that when $n\ge 1$, then we have $2-\frac{n}{2}+\sqrt{n-1}<0$ if and only if $n>10$. So, the dimension $n=10$ is the critical dimension as this is the case for the scalar equation, i. e. $m=1$.  The methods of proof that we apply here are strongly motivated by ideas given by Cabr\'{e}-Capella in  \cite{cc}  and Villegas in \cite{sv}. 

  \begin{thm} \label{radialbound}
 Suppose that $n \ge 2$, $m\ge 1$,  $H\in C^1(\mathbb R^m)$ and $u$ is a radial stable solution of symmetric system (\ref{main}). Then, there exist positive constants $r_0$ and $C_{n,m}$ such that 
 \begin{equation}\label{lower}
 \sum_{i=1}^{m} | u_i(r)| \ge C_{n,m}   \left\{
                      \begin{array}{ll}
                        r^{2-\frac{n}{2}+\sqrt{n-1}}, & \hbox{if $n \neq 10$,} \\
                      \log r, &  \hbox{if $n = 10$,} 
                           \end{array}
                    \right.
                    \end{equation}
                    where $r \ge r_0$ and $C_{n,m}$ is independent from $r$. In addition, assuming that each  $u_i$ is bounded  for $1\le i \le m$, then $n>10$ and  there is a constant $C_{n,m}$ such that 
  \begin{equation}\label{limit}
\sum_{i=1}^{m} | u_i(r) - u_i^\infty| \ge C_{n,m} r^{2-\frac{n}{2} +\sqrt{n-1}}, 
                    \end{equation}
                    where $r\ge 1$ and $u_i^\infty:=\lim_{r\to\infty} u_i{(r)}$ for each $i$.  
 \end{thm}
 
\begin{proof} Suppose that  $u$ is a radial stable solution of  symmetric system (\ref{main}). Then apply  Lemma \ref{stabineqgen} for the following test function $\phi\in H^1(\mathbb{R}^+)\cap L^{\infty} (\mathbb{R}^+)$ 
$$\phi(t):=   \left\{
                      \begin{array}{ll}
                        1, & \hbox{if $0\le t \le 1$;} \\
                       t^{-\sqrt{n-1}}, & \hbox{if $1\le t \le r$;} \\
                        \frac{    r^{-\sqrt{n-1}}   }{  \int_{r}^{R} \frac{dz}{z^{n-1} \sum_{i=1}^{m} {u'_i}^2(z)  } }    \int_{t}^{R} \frac{dz}{z^{n-1}  \sum_{i=1}^{m}  {u'_i}^2(z)  }  , & \hbox{if $r\le t \le R$;} \\
                          0, & \hbox{if $R\le t $,} 
                           \end{array}
                    \right.$$
                    for any $1\le r\le R$.     By straightforward calculations for the given test function $\phi$, the left-hand side of (\ref{Rstablesym}) has the following lower bound,
\begin{equation}\label{LHS}
(n-1)  \int_{0}^{1} \sum_{i=1}^{m} {u'_i}^2(t)   t^{n-3} dt + (n-1)  \int_{1}^{r} \sum_{i=1}^{m} {u'_i}^2(t)   t^{   -2 \sqrt{n-1} +n-3} dt.
\end{equation}
On the other hand, since 
$$\phi'(t)=   \left\{
                      \begin{array}{ll}
                        0, & \hbox{if $0\le t <1$;} \\
                      -\sqrt{n-1} t^{-\sqrt{n-1}-1}, & \hbox{if $1< t < r$;} \\
                     - \frac{    r^{-\sqrt{n-1}}   }{  \int_{r}^{R} \frac{dz}{z^{n-1} \sum_{i=1}^{m} {u'_i}^2(z)  } }     \frac{1}{t^{n-1}  \sum_{i=1}^{m}  {u'_i}^2(t)  }  , & \hbox{if $r\le t \le R$;} \\
                          0, & \hbox{if $R\le t $,} \\
                           \end{array}
                    \right.$$           
the right-hand side of (\ref{Rstablesym}) is the same as the following
\begin{equation}\label{RHS}
(n-1) \int_{1}^{r}   t^{  -2\sqrt{n-1}+ n-3}  \sum_{i=1}^{m} {u'_i}^2(t)  dt +  \frac{    r^{-2\sqrt{n-1}   } }{  \int_{r}^{R} \frac{dz}{z^{n-1}\sum_{i=1}^{m} {u'_i}^2(z)   }  } .
\end{equation}
 Hence, equating (\ref{LHS}) and (\ref{RHS}) we get  
\begin{equation}\label{LR}
\int_{r}^{R} \frac{ds}{s^{n-1}\sum_{i=1}^{m} {u'_i}^2(s)   }  \le C_{n,m}  r^{-2\sqrt{n-1}}   \ \ \ \ \forall 1\le r\le R,
\end{equation}
where  $C_{n,m}:=\left((n-1)  \int_{0}^{1} \sum_{i=1}^{m} {u'_i}^2(t)   t^{n-3} dt\right)^{-1} $. Note that the constant $C_{n,m}$ does not depend on $r,R$.   Applying the H\"{o}lder's inequality we obtain 
\begin{eqnarray}\label{l=r}
\ \ \ \ \ \  \int_r^R \frac{ds}{s^{\frac{n-1}{3}}} &= &\int_r^R \frac{\left(  \sum_{i=1}^{m} {u'_i}^2(s)   \right)^{1/3} }{s^{\frac{n-1}{3}} \left(  \sum_{i=1}^{m} {u'_i}^2(s)   \right)^{1/3}  } ds  
\\ & \le& \nonumber
 \left( \int_r^R  \frac{ds}{  s^{n-1}   \sum_{i=1}^{m} {u'_i}^2(s)   }         \right)^{1/3}   \left( \int_r^R   \left(\sum_{i=1}^{m} {u'_i}^2(s)     \right)^{1/2}    ds  \right)^{2/3} .   
\end{eqnarray}
From (\ref{LR}) and the fact that $||z||_{l^2}\le ||z||_{l^1}$ for any $z\in \mathbb R^m$, we have
\begin{equation}\label{l=r2}
\int_r^R \frac{ds}{s^{\frac{n-1}{3}}} \le C_{n,m} r^{-\frac{2}{3}\sqrt{n-1}}    \left(  \sum_{i=1}^m \int_r^R    |u'_i(s) | ds \right)^{2/3} .
\end{equation}
 Computing the integral in the left-hand side of (\ref{l=r2}) and taking $R=2r$, for any $n\ge 2$,  we get 
\begin{equation}\label{u2r}
 \sum_{i=1}^m | u_i(2r)-u_i(r)| \ge C_{n,m} r^{2-\frac{n}{2}+\sqrt{n-1}}.
\end{equation}
Finally assuming that $u$ is bounded, from (\ref{u2r}) we conclude 
 \begin{equation}\label{}
\sum_{i=1}^{m} |u_i(r)- u_i^\infty| = \sum_{i=1}^m \sum_{k=1}^\infty  | u_i(2^k r)-u_i(2^{k-1}r)| \ge C\sum_{k=1}^\infty  (2^{k-1}r)^{2-\frac{n}{2}+\sqrt{n-1}}.
\end{equation}
 This proves the second part of the theorem that is (\ref{limit}) and $n>10$.  To prove the first part of the theorem that is (\ref{lower}),  without loss of generality,  we assume that $2\le n\le 10$.    Define $r=2^{k-1}r_1$ where $1\le r_1<2$. Therefore, 
\begin{eqnarray*}\label{}
\sum_{i=1}^{m} |u_i(r)| &=& \sum_{i=1}^m  | u_i(r)-u_i(r_1)| - \sum_{i=1}^m |u_i(r_1)| 
\\&=&\sum_{i=1}^m \sum_{j=1}^{k-1}  | u_i(2^j r_1)-u_i(2^{j-1}r_1)| - \sum_{i=1}^m |u_i(r_1)| \\&\ge&  C_{n,m} \sum_{i=1}^m \sum_{j=1}^{k-1} (2^{j-1}r_1)^{  {2-\frac{n}{2}+\sqrt{n-1}}  } -\sum_{i=1}^m |u_i(r_1)|.
\end{eqnarray*}
This shows (\ref{lower}) for the case of $2\le n<10$. From the above inequality, in dimension $n=10$, we have 
\begin{eqnarray*}\label{}
\sum_{i=1}^{m} |u_i(r)| \ge C_{n,m}  (k-1)  - \sum_{i=1}^m |u_i(r_1)|.
\end{eqnarray*}
The fact that $k-1=\frac{\log r - \log r_1}{\log 2}$ finishes the proof. 

\end{proof}

\section{Regularity results for symmetric systems}\label{sReg}
     
In this section, we consider system (\ref{mainom}) when $\Omega=B_1$ where $B_1$ is the unit ball.   Similar to the unbounded case, i. e. (\ref{stabindep}), a stable solution $u=(u_i)_i$ of system (\ref{mainom}) when $\Omega=B_1$  satisfies the following inequality   
  \begin{eqnarray} \label{stablesym}
\nonumber (n-1) \sum_{i=1}^{m} \int_{B_1}  \frac{u'^2_i}{\lambda_i} \phi^2 &\le& \sum_{i=1}^{m} \int_{B_1}  \frac{u'^2_i}{\lambda_i}   |\nabla (r\phi)|^2  \\ &
& +    \sum_{i,j=1}^{m} \int_{B_1}   \left( \partial_j H_i(u)  -\sqrt{\partial_i H_j(u)\partial_j H_i(u)}   \right)  u'_i u'_j (r\phi)^2
\end{eqnarray} 
   for all $ \phi \in C^{0,1}(B_1) \cap H_0^1(B_1)$. In addition, for symmetric systems the following inequality holds 
    \begin{eqnarray} \label{stablesymb}
 (n-1) \sum_{i=1}^{m} \int_{B_1}  \frac{u'^2_i}{\lambda_i} \phi^2 \le \sum_{i=1}^{m} \int_{B_1}  \frac{u'^2_i}{\lambda_i}   |\nabla (r\phi)|^2  
\end{eqnarray} 
   for all $ \phi \in C^{0,1}(B_1) \cap H_0^1(B_1)$.  The fact that $H$ does not appear in this inequality enables us to show that the following regularity result  holds for radial stable solutions of  symmetric system (\ref{main}).   Note that similar results for the scaler case are provided in \cite{v}. 
 \begin{thm} \label{radial}
  Suppose that $n \ge 2$, $m\ge 1$ and $ u=(u_i)_i\in H^1(B_1)$ denotes a radial stable   solution of symmetric system  (\ref{mainom}) where $\Omega=B_1$. Then,  for any $r\in (0,1]$ we have
 \begin{enumerate}
 \item[(i)]  $\sum_{i=1}^{m} \frac{|u_i(r)|}{\sqrt\lambda_i} \le C_{n,m} \sum_{i=1}^{m}\frac{1}{\sqrt{\lambda_i}} ||u_i||_{H^1(B_1\setminus \overline{B_{1/2}})}$,  provided $n<10$,
 \item[(ii)]  $\sum_{i=1}^{m} \frac{|u_i(r)|}{\sqrt\lambda_i} \le C_{n,m}  (1+|\log r|) \sum_{i=1}^{m}\frac{1}{\sqrt{\lambda_i}} ||u_i||_{H^1(B_1\setminus \overline{B_{1/2}})}  $, provided $n=10$,
 \item[(iii)]  $\sum_{i=1}^{m} \frac{|u_i(r)|}{\sqrt\lambda_i} \le C_{n,m} r^{-\frac{n}{2}+\sqrt{n-1}+2}  \sum_{i=1}^{m}\frac{1}{\sqrt{\lambda_i}} ||u_i||_{H^1(B_1\setminus \overline{B_{1/2}})}  $, provided $n>10$,
 \end{enumerate}
 where $C_{n,m}$ is a positive constant independent from $r$. 
\end{thm}
\begin{proof} Let $u=(u_i)_i$ be a radial stable solution of symmetric system (\ref{mainom}). Set the test function $\phi(|x|)$ to be the following for a fixed $r>0$
    \begin{eqnarray*} 
\phi(t)= \left\{ \begin{array}{lcl}
\hfill   r^{-\sqrt{n-1}-1}   \quad && \text{if}\  0\le t\le r,\\ 
\hfill t^{-\sqrt{n-1}-1}   \quad && \text{if}\  r< t\le 1/2, \\
\hfill 2^{\sqrt{n-1}+2} (1-t)  \quad && \text{if}\  1/2< t\le 1.
\end{array}\right.
  \end{eqnarray*}
  Note that the following stability inequality holds,
  \begin{equation}\label{stabphi}
  (n-1) \sum_{i=1}^{m} \int_0^1 \frac{{u'_i}^2(t)}{\lambda_i} \phi^2(t) t^{n-1} dt   \le \sum_{i=1}^{m}  \int_0^1 \frac{{u'_i}^2(t)}{\lambda_i} (t\phi(t))'^2 t^{n-1} dt .
  \end{equation}
  Substitute $ \phi$ into  (\ref{stabphi}) and suppose that $0<r < \frac{1}{2}$.  Then from the fact that $(n-1)\phi^2(t)=(t\phi(t))'^2 $ for $r<t<1/2$ we get  
        \begin{eqnarray} 
\nonumber \int_0^r    \psi_n(t)  \sum_{i=1}^{m}  \frac{{u'_i}^2(t)}{\lambda_i} t^{n-1} dt  & \le& - \int_{1/2}^1 \left(  (n-1)\phi^2(t) -  (t\phi(t))'^2  \right) \sum_{i=1}^{m} \frac{{u'_i}^2(t)}{\lambda_i}  t^{n-1} dt 
 \\&\le & C_n \int_{1/2}^1  \sum_{i=1}^{m} \frac{{u'_i}^2(t)}{\lambda_i}  t^{n-1} dt 
      \end{eqnarray}
where $\psi_n(t):= \left(  (n-1)\phi^2(t) -  (t\phi(t))'^2  \right)  $ and $C_n=|| \psi_n(t)||_{L^\infty   ([1/2,1])}$.  Note that for $t\in [0,r]$ direct calculations show that   $\psi_n(t)=(n-2)  r^{-2\sqrt{n-1}-2}$. Therefore, 
     \begin{eqnarray} \label{bound}
 \int_0^r  \sum_{i=1}^{m}  \frac{{u'_i}^2(t)}{\lambda_i} t^{n-1} dt   \le  C_n r^{2\sqrt{n-1}+2} \int_{1/2}^1  \sum_{i=1}^{m} \frac{{u'_i}^2(t)}{\lambda_i}  t^{n-1} dt 
      \end{eqnarray}
 provided    $0<r < \frac{1}{2}$ and $n>2$.  Similarly one can show that for all $0<r < 1$ and $n \ge 2$, estimate (\ref{bound}) holds by taking the constant $ C_n$ sufficiently large if necessary. From (\ref{bound}) and by a  direct calculation for any $r\in(0,1]$ and $n\ge2$ we get
\begin{eqnarray*} 
\sum_{i=1}^{m} \frac{1}{\sqrt{\lambda_i}} |u_i(r)-u_i(\frac{r}{2})| &\le& \int_{r/2}^{r}  \sum_{i=1}^{m} \frac{1}{\sqrt\lambda_i}  |u'_i(t)|  t^{\frac{n-1}{2}} t^{\frac{1-n}{2}} dt \\ &\le& C_{n,m} \left(  \int_{r/2}^{r} 
 \sum_{i=1}^{m} \frac{1}{\lambda_i}  u'_i(t)^2  
 t^{n-1} dt  \right)^{1/2} \left( \int_{r/2}^{r} t^{1-n} dt  \right)^{1/2} \\&\le& C_{n,m}   r^{\sqrt{n-1}+2-\frac{n}{2}}   \sum_{i=1}^{m} \frac{1}{\sqrt\lambda_i} ||\nabla u_i||_{L^2(B_1\setminus \overline{B_{1/2}})} ,
    \end{eqnarray*}
where $ C_{n,m} $ only depends on $n$ and $m$.  Now, let $0<r \le 1$ so there exist $k\in \mathbb N$ and $1/2<r_1\le 1$ such that $r =\frac{r_1}{2^{k-1}}$. The fact that $u=(u_i)_i$ is a radial solution, we have $|u_i(r_1)|\le ||u_i||_{L^\infty(B_1\setminus \overline{B_{1/2}})} \le C_{i,n} ||u_i||_{H^1(B_1\setminus \overline{B_{1/2}})} $ for all $i=1,\cdots,m$. So, 
\begin{eqnarray*} 
\sum_{i=1}^{m} \frac{1}{\sqrt{\lambda_i}} |u_i(r)| &\le&  \sum_{i=1}^{m} \frac{1}{\sqrt{\lambda_i}} |u_i(r)-u_i(r_1)| + \sum_{i=1}^{m} \frac{1}{\sqrt{\lambda_i}} |u_i(r_1)| \\&\le&   \sum_{i=1}^{m} \frac{1}{\sqrt{\lambda_i}}   \sum_{j=1}^{k-1}   |u_i(\frac{r_1}{2^{j-1}})-u_i(\frac{r_1}{2^j})| +C_{n,m} \sum_{i=1}^{m}\frac{1}{\sqrt{\lambda_i}} ||u_i||_{H^1(B_1\setminus \overline{B_{1/2}})} 
\\&\le& C_{n,m} \sum_{j=1}^{k-1} \left(\frac{r_1}{2^{j-1}}\right)^{-n/2+\sqrt{n-1} +2}     \sum_{i=1}^{m} \frac{1}{\sqrt{\lambda_i}} ||\nabla u_i||_{L^2(B_1\setminus \overline{B_{1/2}})} \\&&+ C_{n,m} \sum_{i=1}^{m}\frac{1}{\sqrt{\lambda_i}} ||u_i||_{H^1(B_1\setminus \overline{B_{1/2}})} 
\\&\le& C_{n,m}  \left(   \sum_{j=1}^{k-1} \left(\frac{r_1}{2^{j-1}}\right)^{-n/2+\sqrt{n-1} +2}   +1     \right)\sum_{i=1}^{m}\frac{1}{\sqrt{\lambda_i}} ||u_i||_{H^1(B_1\setminus \overline{B_{1/2}})} . 
   \end{eqnarray*}
Note that the sign of ${\sqrt{n-1}+2-\frac{n}{2}}$ is crucial in deriving the estimates. Note that  $\sqrt{n-1}+2-\frac{n}{2}=0$ if and only if $n=10$. Therefore, the dimension $n=10$ is the critical dimension.  From the above, for any $0<r\le 1$ we get 
 \begin{equation*} 
\sum_{i=1}^{m} \frac{1}{\sqrt{\lambda_i}} |u_i(r)| \le C_n \sum_{i=1}^{m}\frac{1}{\sqrt{\lambda_i}} ||u_i||_{H^1(B_1\setminus \overline{B_{1/2}})} ,
   \end{equation*} 
   provided  $2\le n<10$.  Note that we have used the fact that $1/2<r_1\le 1$ and  also $\sum_{j=1}^{\infty} \left(\frac{1}{2^{j-1}}\right)^{-n/2+\sqrt{n-1} +2} $ is convergent when $ n<10$.  If $n=10$, then 
   \begin{equation*} 
\sum_{i=1}^{m} \frac{1}{\sqrt{\lambda_i}} |u_i(r)| \le C_n k \sum_{i=1}^{m}\frac{1}{\sqrt{\lambda_i}} ||u_i||_{H^1(B_1\setminus \overline{B_{1/2}})}.
    \end{equation*}
From the definition of $k$ we have $k=\frac{\log r_1-\log r}{\log 2}+1$. Therefore, 
     \begin{equation*} 
\sum_{i=1}^{m} \frac{1}{\sqrt{\lambda_i}} |u_i(r)| \le C_n (1+|\log r|) \sum_{i=1}^{m}\frac{1}{\sqrt{\lambda_i}} ||u_i||_{H^1(B_1\setminus \overline{B_{1/2}})},
    \end{equation*}
where $C_n $ is a large enough constant and depends only on $n$.      Finally, when $n>10$ we have
\begin{equation*} 
 \sum_{j=1}^{k-1} \left(\frac{r_1}{2^{j-1}}\right)^{-n/2+\sqrt{n-1} +2}  =C_n \left(r^{-n/2+\sqrt{n-1} +2} -r_1^{-n/2+\sqrt{n-1} +2} \right). 
    \end{equation*}
Therefore,
\begin{equation*} 
\sum_{i=1}^{m} \frac{1}{\sqrt{\lambda_i}} |u_i(r)| \le C_n r^{-\frac{n}{2} +\sqrt{n-1}+2}  \sum_{i=1}^{m}\frac{1}{\sqrt{\lambda_i}} ||u_i||_{H^1(B_1\setminus \overline{B_{1/2}})} ,
 \end{equation*}
 where $r\in (0,1]$. This finishes the proof. 
      
\end{proof}  
Making an assumption on the sign of the nonlinearity $H(u)$ and its derivatives,  we can prove the following pointwise estimates for derivatives of radial stable solutions of the symmetric system (\ref{main}).  

\begin{thm} \label{highradial}
 Let $\Omega = B_1$, $m\ge 1$ and $n \ge 2$. Suppose that for each $1\le i\le m$, $u_i\in H^1(B_1)$ is decreasing and $ u=(u_i)_i$ is a stable radial solution of symmetric system (\ref{mainom}) where  $H_i(u)\ge 0$.  Then the following estimate holds for $r\in(0,1/2]$, 
      \begin{equation}\label{u'}
\sum_{i=1}^{m} \frac{|u'_i(r)|}{\sqrt\lambda_i}  \le C_{n,m} r^{-\frac{n}{2}+\sqrt{n-1}+1}  \sum_{i=1}^{m}\frac{1}{\sqrt{\lambda_i}} ||  \nabla u_i||_{L^2(B_1\setminus \overline{B_{1/2}})}.
   \end{equation}
Moreover, if $\partial_j H_i(u)\ge 0$ where $i,j=1,\cdots,m$, then 
   \begin{equation}\label{u''}
   \sum_{i=1}^{m} \frac{|u''_i(r)|}{\sqrt\lambda_i}  \le C_{n,m} r^{-\frac{n}{2}+\sqrt{n-1}}  \sum_{i=1}^{m}\frac{1}{\sqrt{\lambda_i}} ||  \nabla u_i||_{L^2(B_1\setminus \overline{B_{1/2}})},
   \end{equation}
 where $C_{n,m}$ is a positive constant independent from $r$. 
\end{thm}
 \begin{proof}  Note that the radial function $u_i$ satisfies $(-r^{n-1} u'_i(r))'=\lambda_i H_i(u)\ge 0$. From this and the fact that $u_i$ is decreasing, we have $-r^{n-1} u'_i(r)$ is positive and nondecreasing. Moreover, the radial function $r^{2n-2} (u'_i(r))^2$ is positive and nondecreasing as well.  So, for any $1\le i\le m$ we get
 \begin{eqnarray*} 
\int_0^{2r}  t^{n-1} ({u'_i}(t))^2 dt &\ge&  \int_r^{2r}  t^{n-1}  ({u'_i}(t))^2 dt  =  \int_r^{2r}  t^{2n-2}  ({u'_i}(t))^2 t^{1-n }dt \\&\ge& r^{2n-2} ({u'_i}(r))^2   \int_r^{2r} t^{1-n }dt \ge C_n  r^{n}   ( {u'_i}(r))^2,  
\end{eqnarray*} 
that gives us the following upper bound 
 \begin{equation*} 
|u'_i(r)| \le C_n r^{-n/2} \left(   \int_0^{2r}  t^{n-1} ({u'_i}(t))^2  dt    \right)^{1/2}.
 \end{equation*}
Taking sum on all values of $1\le i\le m$ we get 
\begin{eqnarray*} 
\sum_{i=1}^{m}   \frac{|u'_i(r)|}{\sqrt{\lambda_i}} &\le& C_n r^{-n/2} \sum_{i=1}^{m}  \left(   \int_0^{2r}  t^{n-1}  \frac{ ({u'_i}(t))^2 }{\lambda_i} dt    \right)^{1/2}\\&\le& \sqrt m  C_n r^{-n/2}  \left(   \int_0^{2r}  t^{n-1} \sum_{i=1}^{m}  \frac{ ({u'_i}(t))^2 }{\lambda_i} dt    \right)^{1/2}.
 \end{eqnarray*}
From this and the estimate (\ref{bound}), in the proof of Theorem \ref{radial},  we get 
\begin{eqnarray*} 
\sum_{i=1}^{m}   \frac{|u'_i(r)|}{\sqrt{\lambda_i}}  &\le&\sqrt m  C_n  r^{-n/2+\sqrt{n-1}+1}  \left(   \int_{1/2}^{1}  t^{n-1} \sum_{i=1}^{m}  \frac{ ({u'_i}(t))^2  }{\lambda_i} dt    \right)^{1/2} \\ &\le &   
 C_{n,m}  r^{-n/2+\sqrt{n-1}+1}   \sum_{i=1}^{m}\frac{1}{\sqrt{\lambda_i}} ||  \nabla u_i||_{L^2(B_1\setminus \overline{B_{1/2}})}.
 \end{eqnarray*}
 This finishes the proof of (\ref{u'}). To prove (\ref{u''}), define $v_i(r)=-nr^{1-1/n} u'_i(r^{1/n})$ for $r\in(0,1]$. It is easy to see that $v'
_i(r)=-\Delta u_i(r^{1/n})=\lambda_i   H_i(u(r^{1/n}))$. Therefore, $v_i$ is a nonnegative nondecreasing function.  Note also that  
\begin{equation*} 
v''_i(r)= \lambda_i \sum_{j=1}^{m} \partial_j H_i(u(r^{1/n})) u_j(r^{1/n}) \frac{r^{1/n-1}}{n} \le 0.
 \end{equation*}
Therefore, $v_i$ is a concave function. This implies that $0\le v'_i(r) \le \frac{v_i(r)}{r}$ for $r\in(0,1]$  that is 
\begin{equation*} 
0\le - u''_i(r^{1/n}) - (n-1)r^{-1/n} u'_i(r^{1/n}) \le -n r^{-1/n} u'_i(r^{1/n}).
 \end{equation*}
Simplifying the above we get 
\begin{equation*} 
r^{-1/n} u'_i(r^{1/n})\le  u''_i(r^{1/n})  \le  -(n-1) r^{-1/n} u'_i(r^{1/n})
 \end{equation*}
 that gives us $|u''_i(r)| \le (n-1)  \frac{|u'_i(r)|}{r}$ for any $r\in(0,1]$.  This finishes the proof of (\ref{u''}).

\end{proof}

\end{document}